\documentclass[BCOR8mm,DIV14]{scrartcl}
\usepackage{umlaut}
\usepackage{bbm}
\usepackage[english]{babel}
\usepackage{amsmath,amssymb,amsthm}
\DeclareMathOperator{\diag}{diag}
\newtheorem{theorem}{Theorem}[]

\newtheorem{corollary}{Corollary}

\newtheorem{fact}{Fact}
\newtheorem{remark}{Remark}
\newtheorem{definition}{Definition}
\begin{document}
 \title{   Another   View   on   the    Hölder Inequality   }
\author{VOLKER  THÜREY  \\  Rheinstr. 91  \\   28199 Bremen, Germany   
     \thanks{T:  49 (0)421/591777    \  \   E-Mail: volker@thuerey.de }   }
    \date{\today}
\maketitle
   \begin{center} Subj-class: FA  \quad \quad 	  MSC-class:  46B, 46G 
                       \quad \quad  Keywords: Hölder inequality, Operator norm     
    \end{center}           
    \begin{abstract}    
    \centerline{Abstract}  
   Every diagonal matrix  {\bf D} yields  an  endomorphism  on  the  $n$-dimensional  complex vector space.
   If one provides     
       the  $\mathbbm{C}^{n}$   with    Hölder norms,  we can compute the operator norm of  {\bf D}. 
    We define homogeneous weighted spaces as a generalization of normed spaces.
   We generalize the  Hölder norms for negative values,  
   this leads to  a  proof of  an extended version of the Hölder inequality. 
   Finally, we formulate this version also for measurable functions.   
	  \end{abstract}       
	    \section{Introduction }  
	  In this paper we generalize the  well-known Hölder inequality  (see, for instance,  \cite{Meise/Vogt}
	  or   \cite{Elstrodt}, or other books on functional analysis). So far nobody discussed
	   the case of negative exponents in all details (for some discussions see e.g. \cite{Mitrinovic/Vasic},p.51).     
     The main reason for this might be the fact that for $ p < 0$   the map $ (x_1, x_2, \ldots , x_n ) \longmapsto 
     \sqrt[p]{|x_1|^{p}+|x_2|^{p}+ \ldots + |x_n|^{p}} $ \ does not yield  a norm for  $ \mathbbm{C}^{n}$,
     because  it is neither positive definit, nor the triangle inequality holds. 
     Although it is worth to consider  this map, since  this leads to a natural extension of 
     the often used   Hölder inequality. 
      To get this result,  we  first  introduce { \it homogeneous weighted spaces } 
      generalizing  normed spaces.  
      Then we define  { \it Hölder weights } as a generalization of the  Hölder  norms, 
      and the { \it operator weight } as a generalization of the operator norm.
      In our first rather inconvenient theorem 
       we compute  the   operator weight of a diagonal matrix. 
    The main result of this paper is then an extension of the Hölder inequality.
     Finally,  we  prove an analogic  result for measurable functions.   But here the proofs rely on the  
       standard Hölder inequality.  \\  \\  
  Let  \ $ X $ \ be a  complex vector space.  \
  Let \   $\|..\|$ \  denote a positive functional on $ X$, that means: \  
   $\|..\|$: \ $ X \longrightarrow$  $\mathbbm{R}^{+}  \cup  \{ 0 , \infty\}$. \   
  We consider three conditions,   \\
  %
       (1)  \ \         $\|  \vec{0}\|$  =  0   \   
                     \rm  and   \  for  all  \it z \ $\in \ \mathbbm{C}$ \ \rm and  all  
                      \it   $\vec{x}  \in X$ \  \rm we have: \
                    \it  $\|z \cdot  \vec{x}\|$ =   $|z|$ $ \cdot  \|\vec{x}\|$ \quad  \hfill \rm ("homogenity"),          \\
       (2) \ \      \rm    \it $ \infty \notin \rm image(\|..\|) $ \rm  \   and
              \  \rm  $\| \vec{x}\| = 0$ { \rm if and only if}  $ \vec{x} = \vec{0}$  
                                                            \hfill  \quad \rm ("positive \ definiteness"), \\
       (3) \  \    \rm  For \ all \ \it  $\vec{x} , \vec{y} \in X$ \ \rm  one \ has \ 
                   \it   $\|\vec{x}+\vec{y}\| \leq  \|\vec{x}\|  + \|\vec{y}\|$ 
                                                             \hfill \quad \rm ("triangle \ inequality").                                         
   \begin{definition}  \rm   \quad  \\                      
    $  \begin{array}[ ]{lll}  \rm  
     \ \ \rm If \  \ \   \|..\|  \  fullfils \ (1) & \rm then  & \ \rm we \ call \ \|..\| \ \  a \  
      { \it homogeneous \ weight }  \ \rm on \ \it X,  \\    
     \ \ \rm if \ \rm \ \   \|..\|  \  fullfils \ (1) , (2) & \rm then &  \ \rm  we \ call \ \|..\| \ \  a \  
      { \it pseudonorm }  \ \rm on \ \it X, and  \\ 
     \ \ \rm if \  \ \   \|..\|  \  fullfils \ (1), (2) \ and \ (3) & \rm then & \  \|..\| \ \  \rm is  \ called 
       \ a \ { \it norm } \rm \ on \ \it X.  
  
                                                              \end{array}   $  
    \end{definition} 
    Acording to this three cases we call                                                           
       the pair \   ( $ X , \|..\|$ ) \ a homogeneously weighted vector space  (or {\bf hw space}), 
       a pseudonormed   vector space, or  a normed vector space, respectively.      
       \begin{definition}  
     \rm For a linear map     $F :  ( X , \|..\|_X ) \ \longrightarrow \ ( Y , \|..\|_Y ) $ \  
          between   complex  homogeneously weighted vector spaces we denote by    \  
   $ \| F \|   :=      \inf \ \{ {\cal C} > 0 \ | \ \forall  \vec{x} \ \in \ X : 
    \| F(\vec{x})\|_{Y} \leq  {\cal C} \cdot  \|\vec{x}\|_{X} \} $  \
    \rm  \ the  { \it operator weight } of  \ $ F $ \ with respect to $ \|..\|_{X}, \ \|..\|_{Y} $.  
     \end{definition}
   Let  {\bf A}  be  a \ 
  complex valued $m\times n$ matrix, \ $ m,n$  $\in$  $ \mathbbm{N}$.  \ 
  Then  {\bf A} defines a linear map, \ \  {\bf A}: $\mathbbm{C}^{n} \rightarrow \mathbbm{C}^{m}$.  \
  Let \ \ $\|..\|_{X} ,  \: \|..\|_{Y}$  \ be homogeneous weights  on 
   $ X := \mathbbm{C}^{n}$ and $ Y := \mathbbm{C}^{m}$,       respectively. \    Then the operator weight is   \ \   
    $\|{\bf A}\| \  =  \  \inf \{ {\cal C} > 0 \ | \ \forall  \vec{x} \in  \mathbbm{C}^{n} : 
    \|{\bf A}\vec{x}\|_{Y} \leq  {\cal C} \cdot  \|\vec{x}\|_{X} \} $. \\  
  This definition turns  \  \{{\bf A} $|$ {\bf A}: ($\mathbbm{C}^{n},\|..\|_{X})  \rightarrow                                                (\mathbbm{C}^{m},\|..\|_{Y}) $  \  and \  {\bf A} is linear\}  \  into
   a  {\bf hw space}, which is a pseudonormed space, or a normed space,  respectively, 
   depending on the properties of the homogeneous weights \ $ \|..\|_{X}$   \ and \ $\|..\|_{Y}$.  \\
  Now  \ for every \ $n$ $\in$  $ \mathbbm{N}$ \ 
  \ and for every  \ $p$ $\in$ \  $ \{\infty, -\infty \}   \cup \: \mathbbm{R}\backslash \{0\} $  \ \ 
  we construct a homogeneous weight on $ \mathbbm{C}^{n}$.  
    \begin{definition}  \quad   \\    \rm
     For  $\vec{x} = (x_1 , \ \ldots \ ,x_n ) \in  \mathbbm{C}^{n} $ \
    and for \   $p$ $\in$  $\left( 0, \infty \right)$ \ set   \  
   $\|\vec{x}\|_{p}  := \sqrt[p]{|x_1|^{p}+|x_2|^{p}+ \ \ldots \ + |x_n|^{p}}$,  \\
   and    for  $p$ $\in$  $\left( -\infty , 0 \right)$ we set  
   \rm  \[   \|\vec{x}\|_{p}  :=             
             \begin{cases}
           \sqrt[p]{|x_1|^{p}+|x_2|^{p}+ \ \ldots \ + |x_n|^{p}}  \quad  &
       \quad \mbox{for} \quad       \prod_{i=1}^{n} x_i \neq 0   \\    
         0   &  \quad \mbox{for} \quad        \prod_{i=1}^{n} x_i = 0   \\            
                         \end{cases}
                                     \]   
   \rm    and  \  \ $\|\vec{x}\|_{\infty} := \max \{ \: |x_i|\: \   | \ i \in \{1,2, \ldots , n\} \  \} , $   \ \
          and  \  \ $\|\vec{x}\|_{-\infty} := \min \{ \: |x_i|\: \   | \ i \in \{1,2, \ldots , n\}  \ \} $. 
     These homogeneous weights will be called  the { \it Hölder weights }  on   $\mathbbm{C}^{n}$.     
   \end{definition}
    \begin{remark}  \rm 
     Note that for $p$ $<$ 0 \ we have \  $\|\vec{x}\|_{p} =  0 \ \Longleftrightarrow \ 
     \exists \ j \in \ \{ 1, \ldots , n \} $
                               \rm   \  and  \ $x_j = 0  $.  
        Furthermore, for all  $n$ $>$ 1, these  Hölder weights   
      are  pseudonorms if and only if \ $p$ $>$ 0, \ and they are  norms if and only if \ \ $p$ $\geq$ 1. \
     \end{remark}    
    In the case of a diagonal matrix \ {\bf D}, \   
    $ {\bf D} : \ ( \mathbbm{C}^{n},\|..\|_{s})$  $\rightarrow$ ($\mathbbm{C}^{n},\|..\|_{t})$  \ 
     and  \    $\|..\|_{s} , \|..\|_{t}$ \ \  are Hölder weights,  one easily verifies that 
     \[   \|{\bf D}\| 
      = \sup \ \{ \|{\bf D}\vec{x}\|_{t} \: | \:  \vec{x} \in  \mathbbm{C}^{n} \
                    \begin{rm}   and  \end{rm} \    \|\vec{x}\|_{s} = 1 \} . \] 
    This  equality does not hold in general for arbitrary  linear maps \  
     F :  (X,$\|..\|_X$ ) \ $\longrightarrow$ \ ( Y,$\|..\|_Y$) \ due to the fact that there need not to
    exist an $\vec{x}$ with  $  \|\vec{x}\|_{X} = 1 $.

     Let us now restrict our attention to diagonal matrices to state  our first theorem.     
      \begin{theorem}
   For \ $ n \geq $ 2 \  and \ 
              $\vec{v} := ( v_1,  \ldots , v_n ) \in  \mathbbm{C}^{n} $    \  \  \  
     let      $   \begin{array}[ ]{ccc}
                    \quad    \  {\bf D}  :=  \  &
                                                     \left(   \begin{array}[ ]{ccc}
                                                                         v_1  &         &  {  \bf 0 }   \\
                                                                              &  \ddots &  \\
                                                               {  \bf 0 }     &         & v_n   	
                                                              \end{array}              \right)           &        
              \end{array}     $                   
   be  the  associated  n-dimensional diagonal matrix, and let    \   
   $ s, t \in \  \mathbbm{R} \backslash \{0\} \cup \{+\infty , \: -\infty \}$. 
    Thus \  {\bf D} \    
       is a linear    endomorphism on    $\mathbbm{C}^{n}$. 
      Then we have for the operator weight \ $ \|{\bf D}\| $ \ with respect to \ $ \|..\|_{s} \ and \ \|..\|_{t}$                      
    \[
       \|{\bf D}\|_{s,t} :=   \| {\bf D}\| =             
             \begin{cases}
          \infty                     \ &
          { \rm  if }   \quad  
        ( s < 0 < t  \ \wedge    \    \vec{v} \neq \vec{0} )  \hfill \qquad  (\mathbb{A}), \\                
       \|\vec{v}\|_t  \  &
          { \rm  if }  \quad   ( \vec{v}  =  \vec{0} ) \ \ \ \vee \ \ \ ( t < 0   \wedge  \prod_{i=1}^{n} v_i =  0 ) 
    \ \ \  \vee \ \ \  (  s = \infty )   \hfill \qquad  (\mathbb{B}), \\     
       \|\vec{v}\|_{\frac{s \cdot t}{s-t}}  \  &
           { \rm  if }  \quad       
         ( -\infty < t < s < 0 \ \wedge \ \prod_{i=1}^{n} v_i \neq 0 )   \   \quad    \vee  \\  & \quad \quad \quad             ( 0 < t < s < \infty ) \ \  \vee  \ \   \ ( -\infty < t < 0 < s < \infty ) \hfill \qquad   (\mathbb{C}),   \\
           \|\vec{v}\|_{\infty}    \ &
         { \rm  if }  \quad       
      ( \ s \leq t < 0 \ \wedge \ \prod_{i=1}^{n} v_i \neq 0 ) \ \  \vee \ \ 
                                                              ( 0 < s \leq t ) \hfill \qquad  (\mathbb{D}), \\
         \|\vec{v}\|_{-s}  \  &
         { \rm  if }  \quad     ( \ t = -\infty \ \wedge \ \prod_{i=1}^{n} v_i \neq 0 ) \hfill \qquad   (\mathbb{E}).                             \end{cases}
                                     \]        
 \end{theorem}
  Note that all possible cases are covered  by $(\mathbb{A}) - (\mathbb{E})$. \
  The above theorem allows us to deduce a theorem and two corollaries.   
  \begin{corollary}
     Let  \  $ s, t \in \ \mathbbm{R} $ \ \ such that \ \  
    $  0  \neq s \cdot t $, \ \ and \   for \ 
      $  {\bf D} := \diag (v_1, \ldots, v_n)$   with \ $ \prod_{i=1}^{n} v_i \  \neq \ 0 $  \ \ \
      we have     \[   \|{\bf D}\|_{s,t}  \ = \  \|{\bf D}\|_{-t,-s} \ . \]  
     \end{corollary} 
    \begin{theorem}  { \bf  [Generalized Hölder Inequality] } \\
      Let \ \ r, s, t $\in$  $\mathbbm{R}$ \ \ and \ \ $ 0 \neq$ $r \cdot s \cdot t$  \  
                        \  and \ \ $\frac{1}{t}$ = $\frac{1}{r} \ + \ \frac{1}{s}$ . \  \
              Then we have  \ \ for every  \ \ n $\in$  $ \mathbbm{N}$  \\   
              for all vectors  \ \  $ \vec{v} := (v_1, \ \ldots , \ v_n)  \: $ \ and \
                $ \vec{x} := (x_1, \ \ldots , \ x_n)  \  \in  \mathbbm{C}^{n} $ \ 
       $ ( with  \ \  \vec{v} \cdot \vec{x} $  \ \  denotes  multiplication by components $) $       
                $$ \ t < r , s  \quad \Longrightarrow  \quad     
                                \| \vec{v} \cdot \vec{x} \|_t    \leq 
                              \| \vec{v}  \|_r  \cdot   \| \vec{x} \|_s \ , $$    
                      $$ \  t > r , s  \quad \Longrightarrow  \quad     
                               \| \vec{v} \cdot \vec{x} \|_t    \geq 
                              \| \vec{v}  \|_r  \cdot  \| \vec{x} \|_s  \  . $$                                           
   \end{theorem} 
  More explicitely  we have the following corollary. 
    \begin{corollary}  { \bf  [Generalized Hölder Inequality] } \\
      Let \quad    r, s, t $\in$  $\mathbbm{R}$ \ \ such that  \ \  $ 0 \neq$ $r \cdot s \cdot t$  \ \ 
                          and \ \  $\frac{1}{t}$ = $\frac{1}{r} \ + \ \frac{1}{s}$ . \  \
              Then for every   \ n $\in$  $ \mathbbm{N}$  \ and  
              for all numbers  \ \ $ v_1, \ \ldots , \ v_n , \
                              x_1, \ \ldots , \ x_n   \  \in  \mathbbm{C} $  
                              \ \   with   \ \ $\prod_{i=1}^{n} \ v_i \cdot  x_i$  $\neq 0 $ \ \ we have   
                 $$  \ t < r , s  \quad \Longrightarrow  \quad     
                               \sqrt[t]{\sum_{i=1}^{n}|v_i\cdot x_i|^{t}} \leq 
                             \sqrt[r]{\sum_{i=1}^{n}|v_i|^{r}} \cdot  \sqrt[s]{\sum_{i=1}^{n}|x_i|^{s}} \ , $$  
                  $$  \  t > r , s  \quad \Longrightarrow  \quad     
                               \sqrt[t]{\sum_{i=1}^{n}|v_i\cdot x_i|^{t}} \geq 
                             \sqrt[r]{\sum_{i=1}^{n}|v_i|^{r}} \cdot  \sqrt[s]{\sum_{i=1}^{n}|x_i|^{s}} \ . $$  
              \end{corollary}  
     \begin{remark}  \rm          
     \quad If \ \ $\prod_{i=1}^{n} \ v_i \cdot  x_i$ = 0 \ \
        the inequality remains true provided the roots for negative exponents are defined.
    \end{remark}  
       \section{ Proof of Theorem 1 }     
        First we handle  the two easy cases.   \\
  CASE   $(\mathbb{A})$.  \ 
    Let   \ \   $s < 0 < t$ \ and  \ $ \vec{v} \ \neq \ \vec{0}$. \\
  Because  ${\bf D}$ \ is  not   the   0-matrix, there  is  a  \  
                 $ j \in \{ 1, \ldots , n \}  $ \  with  \  $ \ v_j \neq 0$.  \
  Take \   for every \ \ $k$ $\in \mathbbm{N}\backslash\{1\}$  \ the vector \ 
  $\vec{a}_k$ := ($ a_{k,1}, \:  \ldots , \: a_{k,n} $) \  with \  $a_{k,j}$ := $k$ \  and  \   
   for all  \ $ i   \in$ $\{1,2, \ldots , n \}\backslash \{j\}$  \  let \ 
  $a_{k,i}$ :=  $ \sqrt[s]{\frac{1-k^{s}}{n-1}}$ .  \ 
   We have   for every \  $k$ $\in \mathbbm{N}\backslash\{1\}$:   \quad $\|\vec{a}_k\|_s$ = 1  \  
  and \  $\|{\bf D}(\vec{a}_k)\|_t$ $\geq$ $|k \cdot v_j|$,   \  and    because of   \  
  $ k \rightarrow \infty$  \  the right hand side  goes to  infinity, hence  \  $\|{\bf D}\|_{s,t}$ = $\infty$.
                        \\ 
  
  CASE   $(\mathbb{B})$.  \ 
    Let \   \ $ \vec{v}  =  \vec{0}$, \ \ or \ \ $ t < 0 $\  and \ $ \prod_{i=1}^{n} v_i =  0 $,
   \ \ or \ \  $ s = \infty $.  \\
   If \    {\bf D} \  is \ the \  0-matrix \  we have  \ for all \ $s,t$ : \ $\|{\bf D}\|_{s,t}$  = 0 . \   
   If \  $( t < 0   \wedge  \prod_{i=1}^{n} v_i =  0 )$ \ \ one has at least  one  
                                                          \ \ $ j  \in \{1, \ldots , n \}$ \ \ with \ 
   \ $  v_j = 0$. \
  Then \ for $ \  \ \vec{x} \in$  $\mathbbm{C}^{n}$  \  
  we have \ \  $v_j \: x_j = 0$, \   hence \    $\|{\bf D}(\vec{x})\|_t$ = 0,
                                \ hence \   $\|{\bf D}\|_{s,t}$ = 0 =  $\|\vec{v}\|_{t}$ .  \\
     In the case of  \ $s = \infty $ \
 take \   $\vec{e} := (1,1, \ldots , 1 )$, \  then we have \  
                 $ \|\vec{e} \: \|_{\infty}$ = 1.     \   If \  $ t \in   \mathbbm{R} $
  \  we get \   $ \|{\bf D} \vec{e} \: \|_{t}  =  \left[   \sum_{i=1}^{n} |v_i|^{t} \right]^{\frac{1}{t}}$.  \ 
   If \  $ t = -\infty $ \  we get \  
    $  \|{\bf D} \vec{e} \: \|_{-\infty} =  \min \{ \: |v_i|\: \   | \ i \in \{1,2, \ldots , n\}  \ \}$.    \
     Hence  in  CASE $(\mathbb{B})$ we  always have   
                          \quad  $\|{\bf D}\|_{s,t}$  =   $\|\vec{v}\|_{t}$ .   \\    
    \\      The following two cases  are more complicated and  they  need  more attention. 
     They will be treated  together, because the proofs are similar.   \\ 
   CASE   $(\mathbb{C})$  \ and \ CASE    $(\mathbb{D})$. \quad    Let either  \ \ 
         $   ( -\infty < t < s < 0 \ \wedge \ \prod_{i=1}^{n} v_i \neq 0 ),  \ \  {\rm or } \ \  
         { ( 0 < t < s < \infty ) }, 
                  \\ {\rm or} \ \ ( -\infty < t < 0 < s < \infty ),   \ \ {\rm or} \ \ 
        \ ( \  s \leq t < 0 \ \wedge \ \prod_{i=1}^{n} v_i \neq 0 ), \ \  {\rm or}  \ \  ( 0 < s \leq t )$.   \\  
   The theorem is trivial if {\bf D} is the 0-matrix, because  then it clearly follows that \
    0  =   $ \|\vec{v}\|_{\frac{s \cdot t}{s-t}} $ = $ \|\vec{v}\|_{\infty} $ =  $ \|{\bf D}\|$. 
       Hence, we assume \   $ \vec{v}  \neq \vec{0} $.                        Let    \
 $\sf M$ $\in$ \{ $1,2, \ \ldots , \ n$ \} \    such that   \  
   $ |v_{\sf M}| $ =  { $ \max \{|v_1|,  \ldots , |v_n| $\}, } hence $ |v_{\sf M}| $ $>$ 0.   \ \
                       Now for the proof we will distinguish four different cases.   \\ 
      {\large \bf Case a)} \  $  0 <   s, t  < \infty $.    \\  
      {\large \bf Case b)} \    $-\infty < s , t < 0 \   { \rm  and } \   \prod_{i=1}^{n} v_i \neq 0 $.  \\
      {\large \bf Case c)} \ $-\infty < t < 0 < s <   \infty $.   \\ 
      {\large \bf Case d)} \  $  ( \ -\infty = s \leq t < 0 \ \ {\rm  and}  \ \  \prod_{i=1}^{n} v_i \neq 0 \ ) \ \ 
                          {\rm  or}  \ \ ( \  0 < s \leq t = \infty \ ) $.  \\ 
      We will prove  the  cases {\bf a,b,c } for \ $n = 2$ \   \ 
                          and then inductively  for all \ $ n \in$  $ \mathbbm{N}\backslash\{1\}$.     \\     \\                   
       {\large \bf Case a)} \quad Let  \   0 $<  s, t < \infty $.   \\  
          Let  \ $ n = 2$.    \ \
        We have the $ 2\times2$ matrix \quad  {\bf D} \ := \  $ \diag (v_1,v_2)$.              
     Without loss of generality let \   $v_{\sf M}$ = $ v_2$   ( $\neq$ 0 ).  \ 
                                  With \    $ b := {v_1}/{v_2}$  \  
     we have  \ \ $|b|$ $\leq$ 1, \   and \   {\bf D}   \   = \   
       $ v_2$ $\cdot$   $  \left(   \begin{array}[ ]{cc}
                     b  &    0  \\
                    0   &    1   	
            \end{array} \right) $ \ \ =: \ $ v_2$ $\cdot$ $\widetilde{{\bf D}}$.  \\
       We have \   $\left\|{\bf D}\right\|_{s,t}$ = $|v_2|$ $\cdot$ $\|\widetilde{{\bf D}}\|_{s,t}$ 
        =  $|v_2|$ $\cdot \sup \ \{ \: \|\widetilde{{\bf D}}(\vec{x})\|_t \  | \  \vec{x} \in                                                                           \mathbbm{C}^{2} \ { \rm and } \ \|\vec{x}\|_s = 1$ \}.   \
        With  \ \quad  $ \vec{x}  := (x_1,x_2) $ \quad  
  we define a map  \ \   $G$ : [0,1] $\rightarrow$  $\mathbbm{R}^{+}  \cup  \{ 0 \}$, but at first we will
   consider $G^{\:t}$ because it is easier 
        ( $G$  and $G^{\:t}$ have extremums at the same values ).  \  
      Define \\      $G\:^{t}(y)$ :=  $ (\|\widetilde{{\bf D}}(\vec{x})\|_t)^{t}$ =
                                         $  y^{t}\cdot |b|^{t} + \left[\sqrt[s]{1-y^{s}} \right]^{t}$ \
   \  for \ $ y := |x_1|, \ \|(x_1,x_2)\|_s = 1 $, \ hence \ \ $ y \  \in$ [0,1].  \\
      First assume that \ $ s \neq t  $.  \
    Elemantary analysis  shows that \[  (G\:^{t})'(y_E) = 0   \    \Leftrightarrow      \
                     y_E =    \sqrt[s]{\frac{1}{1+|b|^{\frac{s \cdot t}{t-s}}}} \ . \]
   Instead  of computing    $(G\:^{t})''(y_E)$  we check the boundaries of the domain of $ G $, hence                
  the maximum  \ \  M$_{s,t} :=  \max\:\{\|\widetilde{{\bf D}}(\vec{x})\|_t \ \: | \ \:  \vec{x} \in                                                          \mathbbm{C}^{2} \  \wedge \ \|\vec{x}\|_s = 1 \} \  = \
            \max\:\{ G(y) \ | \ y \in [0,1] $ \}   \
               is contained in the set \  \quad    $\{ \: G(y_E)\:  ,\:  G(0)\: , \: G(1) \: \} \ = 
   \ \{\:  [1+|b|^{\frac{s \cdot t}{s-t}}]^{\frac{s-t}{s \cdot t}} \: , \: 1 \: , \: |b| \: \}$.    \ \
     To determine   M$_{s,t}$ let us now consider the following  three subcases.  \\ 
    Subcase 1: \   $ s < t $  \ $\Longrightarrow$ \  M$_{s < t}$  = 1  \ \ \ and \  \ \ 
                        $\left\|{\bf D}\right\|_{s,t}$ = $|v_2|$ $\cdot$ M$_{s < t}$  = $  |v_2| $.  
    \\                            
    Subcase 2: \  $ s > t $ \  $\Longrightarrow$ \  M$_{s > t}$ = 
                                       $ [1+|b|^{\frac{s \cdot t}{s-t}}]^{\frac{s-t}{s \cdot t}}$  \
     and      $\left\|{\bf D}\right\|_{s,t}$ =   $|v_2|$ $\cdot$ M$_{s > t}$ 
     = $ [\: |v_2|^{\frac{s \cdot t}{s-t}} + |v_1|^{\frac{s \cdot t}{s-t}} \: ]^{\frac{s-t}{s \cdot t}}$. \\                 Subcase 3: \ $ s = t $ \  $\Longrightarrow$ \  By doing similar calculations as just now 
     (in the case $ s \neq t $),
    we get \  
       M$_{s=t}$ \ = \ $G(0)$ \ = \ 1, \ \ hence  \ \  $\left\|{\bf D}\right\|_{s,s}$ =  $ |v_2|$, \ 
         and the theorem  has been proved for $n = 2$.    
   \begin{remark}    \rm     We have  a continuous 
         behaviour of \ $\left\|{\bf D}\right\|_{s,t}$ \  if  \   $ s = t $, \ \  that means          
      \[  \lim_{\breve{t} \nearrow s}      ( \left\|{\bf D}\right\|_{s,\breve{t}}  ) \ = \ 
        \lim_{ \breve{t}  \nearrow s}    ( \: [ \: |v_2|^{\frac{s \cdot \breve{t}}{s-\breve{t}}} + 
            |v_1|^{\frac{s \cdot \breve{t}}{s-\breve{t}}} \: ]^{\frac{s-\breve{t}}{s \cdot \breve{t}}} ) \ = \  
             \|\vec{v}\|_{\infty} \ = \  
            |v_2| \ = \  \left\|{\bf D}\right\|_{s,s} \ = \   
            \lim_{ \breve{s}  \searrow t}     (  \left\|{\bf D}\right\|_{\breve{s},t} )  \ .   \]  
     \end{remark}             
   Proof \  for \ $ n \geq $ 3.    \\  
    Assume that the theorem holds for $n - 1 $.    \     
    Let  \ $ {\sf m} \in $ $\{1 , \ldots , n-1\}$ \ with 
                          \ $|v_{\sf m}| := \max \{ |v_1|, \: \ldots , \: |v_{n-1}| $\}, \ \ 
         let  $\vec{x} := ( x_1, \:  \ldots , \: x_{n-1}, x_n )   \in      \mathbbm{C}^{n}$. 
         We distinguish two subcases.  \\
   Subcase 1: \ \ $ s < t $ \  or  \ $ s = t$. \\        
     We have just proved the theorem for $n = 2$,  \  that means that  \
      for arbitrary \ \ $ y_1, y_2, \  w_1, w_2 \ \in \: \mathbbm{C}$  \ we have     \ \
      $  \sqrt[t]{ |w_1 \: y_1|^{t} +  |w_2 \: y_2|^{t} }  \
    \leq \  \max\{|w_1|, |w_2|\} \cdot  \sqrt[s]{|y_1|^{s}+|y_2|^{s}} $ .    \ \
     By the assumption,  we have  \ \ 
   $\sqrt[t]{\sum_{i=1}^{n-1}|v_ix_i|^{t} }$ $\leq$ $|v_{\sf m}| \cdot \sqrt[s]{\sum_{i=1}^{n-1} |x_i|^{s} }$ . \\
         By using the assumption and the theorem for $ n = 2 $,  it follows that 
   \begin{eqnarray*}    
  \|{\bf D}(\vec{x})\|_t  & =  &  \sqrt[t]{\sum_{i=1}^{n-1}|v_ix_i|^{t} + |v_nx_n|^{t}}   \\
        &   \leq \   &
     \sqrt[t]{ |v_{\sf m}|^{t} \cdot  \left[ \sqrt[s]{\sum_{i=1}^{n-1}|x_i|^s}  \right]^{t}  + |v_nx_n|^{t} }  \\   
       &  \leq  &   
    \max \{ |v_{\sf m}| , |v_n| \} \cdot \sqrt[s]{\sum_{i=1}^{n-1}|x_i|^{s}  + |x_n|^{s } }  
     \ \   =  \ \  |v_{\sf M}| \cdot  \|\vec{x}\|_s \:  .  
    \end{eqnarray*}   
    Hence \ \ $\|{\bf D}\|_{s,t}$ $\leq$ $|v_{\sf M}|$.  \\
    The vector \  $\vec{e_{\sf M}} := ( 0, \ldots , 0 , 1, 0, \ldots  , 0 )$  \   shows that 
       \   $\|{\bf D}(\vec{e_{\sf M}})\|_t$  = $|v_{\sf M}|$ $\cdot$ 1, \   
    hence \  $\|{\bf D}\|_{s,t}$ = $|v_{\sf M}|$.   \\
    Subcase 2: \ \ $ s > t $. \\
    Let \ \ $\widetilde{w}$ := $ [\: \sum_{i=1}^{n-1}|v_i|^{\frac{s \cdot t}{s-t}} ]^{\frac{s-t}{s \cdot t}}.$ \
    Because the theorem holds for $ n = 2$,  
      \ we have for arbitrary \  $ y_1, y_2, \  w_1, w_2 \in \: \mathbbm{C}$:  \ 
      $  \sqrt[t]{ |w_1 \: y_1|^{t} +  |w_2 \: y_2|^{t} }$  \
    $\leq$ \   [$ |w_1|^{ \: \frac{s t}{s-t}} + |w_2|^{\frac{s t}{s-t}} ]^{\frac{s-t}{st}}$ $\cdot$ 
    $  \sqrt[s]{|y_1|^{s}+|y_2|^{s}} $.    \\
      Because we assume the theorem for $ n-1 $, \ we have :  \ \ 
   $\sqrt[t]{\sum_{i=1}^{n-1}|v_ix_i|^{t} }$ $\leq$ $ \widetilde{w} \cdot \sqrt[s]{\sum_{i=1}^{n-1} |x_i|^{s} }$ . \\
      By using this and  the theorem for $ n =  2$,   we have  
     \begin{eqnarray*}     
         \|{\bf D}(\vec{x})\|_t   & \   =  \ & 
     \sqrt[t]{    \sum_{i=1}^{n-1}|v_i x_i|^{t}   + |v_nx_n|^{t} }   \\
       &  \   \leq  \  & 
    \sqrt[t]{ \widetilde{w}^{ \: t} \cdot  \left[ \sqrt[s]{\sum_{i=1}^{n-1}|x_i|^s}  \right]^{t}  + |v_nx_n|^{t} } \\         
       &  \   \leq   \   & 
    [ \widetilde{w}^{ \: \frac{s t}{s-t}} + |v_n|^{\frac{s t}{s-t}} ]^{\frac{s-t}{st}} \cdot \|\vec{x}\|_s 
    \ \  =   \ \  \|\vec{v}\|_{\frac{s \cdot t}{s-t}} \cdot \|\vec{x}\|_s  \ .  
    \end{eqnarray*}       
       Hence  $\|{\bf D}\|_{s,t}$     $\leq$   $\|\vec{v}\|_{\frac{s \cdot t}{s-t}}$  . \\ 
       Define  \ for all \ $ i = 1,2, \ldots n \quad  r_i := \sqrt[s-t]{|v_i|^{t}}$, \ 
        and  take the vector \  
  $\vec{z}$  := $ \frac{1}{\sqrt[s]{\sum_{i=1}^{n} |v_i|^{\frac{s \cdot t}{s-t}}}}\cdot ( r_1 , \ldots ,  r_n )$. \\
   One has \   $\|\vec{z}\|_s$  = 1  \   and  \   
   $\|{\bf D}(\vec{z})\|_t$ = $\|\vec{v}\|_{\frac{s \cdot t}{s-t}}$ , \ \ 
     that means  the theorem  is satisfied both in  subcase 1 and in subcase 2, and the proof is finished 
    if  \   0 $<  s, t  < \infty $. 
                                                                             \\    \\
   {\large \bf Case b)} \quad Let \ $-\infty  \  < \ s, t \ <$ \ 0  \ \ and \ \ $ \prod_{i=1}^{n} v_i \neq 0 $. \\                                                                           
    Let \  \    {\bf D} \  = \             $  \left(   \begin{array}[ ]{cc}
                                            v_1  &    0  \\
                                             0   &   v_2   	
                      \end{array}              \right)   $  \   = \   
       $ v_2$ $\cdot$   $  \left(   \begin{array}[ ]{cc}
                     b  &    0  \\
                    0   &    1   	
            \end{array} \right) $ \ \ =: \ $ v_2$ $\cdot$ $\widetilde{{\bf D}}$  , \ with   \ 
     $b$ := $v_1/v_2$, as above,  \ and we have \quad   
      $|v_2|$ $\geq$  $|v_1|$ $>$ 0  \   and \   1 $\geq$ $|b|$ $>$ 0.   \ \     
   One has \  $\left\|{\bf D}\right\|_{s,t}$ =  
        $|v_2|$ $\cdot \sup \{\|\widetilde{{\bf D}}(\vec{x})\|_t \ \ | \ \ \vec{x} \in                                           \mathbbm{C}^{2}  \  \ \wedge \ \|\vec{x}\|_s = 1$\}, \  as above.  \\
          But the domain of the map \ \ 
   $ G\:^{t}(y) := (\|\widetilde{{\bf D}}(\vec{x})\|_t)^{t}$  = 
      $ y^{t}\cdot |b|^{t} + \left[\sqrt[s]{1-y^{s}} \right]^{t}$  \ \ 
  has changed.  \\  With  \ \ $ \vec{x} =: (x_1,x_2) $ \ and \ $ \|\vec{x}\|_s = 1 $, \
   $ y := |x_1| $, it has to be \  $ y >$ 1, 
    ( because  \ $ s $ \ is negative ).  \\      
     As above, we have  \quad    $(G\:^{t})'(y_E) = 0   \    \Leftrightarrow      \
    y_E =  \sqrt[s]{\frac{1}{1+|b|^{\frac{s \cdot t}{t-s}}}}$ \ = \
     $[{1+|b|^{\frac{s \cdot t}{t-s}}}]^{-\frac{1}{s}}$ ,  \\     
  and the maximum  \ \ M$_{s,t}$ :=  $ \sup\:\{\|\widetilde{{\bf D}}(\vec{x})\|_t \ \: | \ \:  \vec{x} \in                                                          \mathbbm{C}^{2} \  \wedge \ \|\vec{x}\|_s = 1$ \} \  = \
          $  \sup\:\{ G(y) \ | \ y > 1  \} $  \\         
     is contained in the set \ 
       \quad    $\{ \: G(y_E) , \:   \lim_{ y \to 1} G(y) \: , 
        \lim_{ y \to \infty}   G(y) $  \} \ = 
   \ $\{\:  [1+|b|^{\frac{s \cdot t}{s-t}}]^{\frac{s-t}{s \cdot t}} \: , \: |b| \: , \: 1 \: \}$.  \\ 
   Again we consider three subcases. \\
    Subcase 1: \ \ $ s < t $ \  $\Rightarrow$ \  M$_{s < t}$  = 1  \ \  and   \ \ 
                               $\left\|{\bf D}\right\|(s,t)$ = $|v_2|$ $\cdot$ M$_{s < t}$  = $  |v_2| $.   \\
   Subcase 2: \ \ $ s > t $ \ \ $\Rightarrow$ \ M$_{s > t}$ = 
                               $ [1+|b|^{\frac{s \cdot t}{s-t}}]^{\frac{s-t}{s \cdot t}}$   \   
                               and \ $\left\|{\bf D}\right\|(s,t)$ =   $|v_2|$ $\cdot$ M$_{s > t}$ 
     = $ [\: |v_2|^{\frac{s \cdot t}{s-t}} + |v_1|^{\frac{s \cdot t}{s-t}} \: ]^{\frac{s-t}{s \cdot t}}$. \\  
                   Subcase 3: \ $ s = t $  \ \ $\Rightarrow$ \  We get \  
       M$_{s=t}$ \ = \  
                       $ \lim_{ y \to \infty}  G(y)$ \ = \ 1,
  \  hence  \  $\left\|{\bf D}\right\|(s,t)$ =  $ |v_2|$, and the theorem   has been proved for $ n = 2$. \ 
    Now we finish  {\large \bf Case b}  in a similar way to  {\large \bf Case a}.   \\
       Subcase 1: \ \ $ s < t $ \ or  \ $ s = t $. \\  
    We have proved  the theorem for $ n = 2 $.    Because of  \ $ t < 0 $, \ we have for arbitrary \\
     $ y_1, y_2, \  w_1, w_2  \in   \mathbbm{C} $:    \
       $|w_1 \: y_1|^{t} +  |w_2 \: y_2|^{t} 
    \geq [max\{|w_1|, |w_2|\}]^{t} \cdot \left[ \sqrt[s]{|y_1|^{s}+|y_2|^{s}} \right]^{t}$.
     \  Let  \\  
                             {\sf m} $\in$ $\{1 , \ldots , n-1\}$ \ with
                    \ $|v_{\sf m}|$ := max \{$ |v_1|, \: \ldots , \: |v_{n-1}| $ \}, \quad 
         let  $\vec{x} := ( x_1, \:  \ldots , \: x_{n-1}, x_n )   \in      \mathbbm{C}^{n}$.   \\
     We assume the theorem for $ n - 1 $, hence  we have :  \ \ \
   $\sqrt[t]{\sum_{i=1}^{n-1}|v_ix_i|^{t} }$ $\leq$ $|v_{\sf m}| \cdot \sqrt[s]{\sum_{i=1}^{n-1} |x_i|^{s} }$ . \\
   Because of  \ \ $t <$ 0 , \ this is equivalent to \ 
    $\sum_{i=1}^{n-1}|v_ix_i|^{t}$ $\geq$ $|v_{\sf m}|^{t} \cdot
                                   \left[ \sqrt[s]{\sum_{i=1}^{n-1} |x_i|^{s}} \right]^{t}$. 
   \begin{eqnarray*}        { \rm   Thus \ it \ follows }  \  \quad  
              \left[\|{\bf D}(\vec{x})\|_t\right]^{t}  &  =   &  \sum_{i=1}^{n-1}|v_ix_i|^{t} + |v_nx_n|^{t}  \\
                &  \geq    &
               |v_{\sf m}|^{t} \cdot  \left[ \sqrt[s]{\sum_{i=1}^{n-1}|x_i|^s}  \right]^{t}  + |v_nx_n|^{t}  \\  
             &   \geq   &   
              \left[ \: \max\{ |v_{\sf m}| , |v_n| \} \: \right]^{t} \cdot 
              \left[\sqrt[s]{\sum_{i=1}^{n-1}|x_i|^{s}  + |x_n|^{s }}\right]^{t} 
             = |v_{\sf M}|^{t} \cdot  \|\vec{x}\|_s  ^{t} \ .   
      \end{eqnarray*}   
    Because of \ $ t <$ 0, this is equivalent to \ \  $\|{\bf D}(\vec{x})\|_t$  
                                             $\leq$    $|v_{\sf M}| \cdot  \|\vec{x}\|_s$ .  \
    Hence \ \ $\|{\bf D}\|_{s,t}$ $\leq$ $|v_{\sf M}|$.    \\ 
    To check equality, take  for all sufficient large \ $ k  \in$  $ \mathbbm{N}$  \ \ 
    ( i.e. such that  \ \ \ $ 2 - (1-\frac{1}{k})^{s} \  > \ 0$ ) \\
    the vector \ \
     $\vec{a_k}$ := ($ a_{k,1}, \:  \ldots , \: a_{k,n} $) \ \    with \ \  
     $a_{k, {\sf M}}$ := $ \sqrt[s]{2 - (1-\frac{1}{k})^{s}}$ ,  \ and  \   
   for  every  \\   $  i \in$ $ \{1,2, \ldots , n \}\backslash \{\sf M\} $ \ take \
  $a_{k,i}$  :=  $q_k$  :=  $ \sqrt[s]{\frac{( 1-\frac{1}{k} )^{s} \ -1}{n-1}}$ .  \  
   We have  for all such \ $ k $ :  \  $\|\vec{a_k}\|_s$ = 1, \ \   and  
   because of \ \ $ s, t <$ 0, \ \  we  get \  \     
    $  \lim_{ k \to \infty}$ $(q_k)$ = $\sqrt[s]{0}$ = +$\infty$, \quad
    hence \ \     $ \lim_{ k \to \infty}$ $((q_k)^{t})$ = 0,  
    \begin{eqnarray*}  
     {\rm  and } \ \quad  
      \|{\bf D}(\vec{a_k})\|_t   & = &      \sqrt[t]{ |v_{\sf M}|^{t} 
         \cdot   \left[ \sqrt[s]{2 - (1-\frac{1}{k})^{s}} \: \right] ^{t} + 
                             \sum_{i=1, \ldots , n \wedge i \neq {\sf M}} |v_i|^t \cdot (q_k)^{t}  }   \\ 
                                 &  = & \ \   |v_{\sf M}| \cdot  
      \sqrt[t]{  \left[ \sqrt[s]{2 - (1-\frac{1}{k})^{s}} \: \right]  ^{t} + \ \ (q_k)^{t} \  \cdot 
                             \sum_{i=1, \ldots , n \wedge i \neq {\sf M}} |\frac{v_i}{v_{\sf M}}| ^{t} } \quad ,
      \end{eqnarray*}                         
                                   hence \ \ $\ \lim_{  k \to \infty}$  $\|{\bf D}(\vec{a_k})\|_t$ 
                                                 \ = \ $  |v_{\sf M}|$.  \ \  \quad
       Thus \ \ $\|{\bf D}\|_{s,t}$ = $|v_{\sf M}|$.    \\  
    Subcase 2: \ \  $ s > t $. \\ 
        Let  \quad $\vec{x} := ( x_1, \:  \ldots , \: x_{n-1}, x_n )   \in      \mathbbm{C}^{n}$.  \
    We have proved the theorem for $n = 2$, that means    \\ 
        $|y_1 w_1|^{t} +  |y_2 w_2|^{t}$ $\geq$   
     $ [\: |y_1|^{\frac{s \cdot t}{s-t}} + |y_2|^{\frac{s \cdot t}{s-t}} \: ]^{\frac{s-t}{s}}$  $\cdot$ 
    $\left[ \sqrt[s]{|w_1|^{s}+|w_2|^{s}} \right]^{t}$ \ \ for \ $ y_1, y_2, \  w_1, w_2  \in \mathbbm{C}$.  \\ 
     Assume  the theorem for $n - 1$, hence  (because of \ $ t < 0 $)    \\                              
     $\sum_{i=1}^{n-1}|v_ix_i|^{t}$ $\geq$  $ \left[ \sum_{i=1}^{n-1} |v_i|^{\frac{s \cdot t}{s-t} }                                   \right]^{\frac{s-t}{s}} $   $\cdot$   $\left[ \sqrt[s]{\sum_{i=1}^{n-1} |x_i|^{s}} \right]^{t}$. \ \              By doing similar estimations as three times before, we  get  \ \  
    $\left[\|{\bf D}(\vec{x})\|_t\right]^{t}$  = $\sum_{i=1}^{n-1}|v_ix_i|^{t} + |v_nx_n|^{t}$  \
      $\geq $  \
       $ \left[ \sum_{i=1}^{n-1} |v_i|^{\frac{s \cdot t}{s-t} } \right ]^{\frac{s-t}{s}} $
                  $  \cdot  \left[ \sqrt[s]{\sum_{i=1}^{n-1}|x_i|^s}  \right]^{t}  + |v_nx_n|^{t}$   \ 
      $ \geq $  \  
    $ \left[ \sum_{i=1}^{n-1} |v_i|^{\frac{s \cdot t}{s-t} } + |v_n|^{\frac{s \cdot t}{s-t}} \right ]^{\frac{s-t}{s}}                                                          \cdot 
    \left[\sqrt[s]{\sum_{i=1}^{n-1}|x_i|^{s}  + |x_n|^{s }}\right]^{t} $
    = $\left[ \|\vec{v}\|_{\frac{s \cdot t}{s-t}} \right]^{t} \cdot  \|\vec{x}\|_s  ^{t}$ .   \\
    Because of \ \ t $<$ 0, this is equivalent to \ \   $\|{\bf D}(\vec{x})\|_t$ 
                 $\leq$    $ \ \|\vec{v}\|_{\frac{s \cdot t}{s-t}}  \cdot \  \|\vec{x}\|_s$ .  \\
    Hence \ \ $\|{\bf D}\|_{s,t}$ $\leq$ $\|\vec{v}\|_{\frac{s \cdot t}{s-t}}$ .   
    To check equality , one  can use the same vector as above, i.e. ,
     define \quad  \ for \ \ $i = 1,2, \ldots n : \ \ r_i := \sqrt[s-t]{|v_i|^{t}}$  , \ \  and   \ \   
   $\vec{z}$  := $ \frac{1}{\sqrt[s]{\sum_{i=1}^{n} |v_i|^{\frac{s \cdot t}{s-t}}}}\cdot ( r_1 , \ \ldots , \ r_n )$.
      \\    \\   
      {\large\bf Case c)} \quad Let \ \ $ -\infty < t < 0 < s < \infty.  $  \\    
    The  proof   is  similar as the proofs before and we will  not  explain it in all  details.  \\
    In the case of \ \  $\prod_{i=1}^{n} v_i = 0$, in  \  CASE  $(\mathbb{B})$ \ we already have proved that 
      $ \|{\bf D}\|_{s,t} $ = 0.  
    Note that \quad $\frac{s \cdot t}{s-t}$  $<$ 0, \  \quad hence  \quad [ $\prod_{i=1}^{n} v_i = 0$ 
    \   $\Rightarrow$  \    $  \|\vec{v}\|_{\frac{s \cdot t}{s-t}}$ = 0 ] \quad  follows.  \quad  
     Now assume \ \  $\prod_{i=1}^{n} v_i \neq 0$.   \\
      Proof for \ $ n = 2 $.    \ \quad  As  in   {\large \bf Case a},
        we consider  the $ 2\times2$ matrix \quad  {\bf D} \ := \                  
                 $  \left(   \begin{array}[ ]{cc}
                                            v_1  &    0  \\
                                             0   &   v_2   	
                      \end{array}              \right)   $.   \\ 
     With   \  \ $v_{\sf M}$ = $ v_2$ \ \  and   \ \  $ b := {v_1}/{v_2}$  \  \
     we have  \quad \ \ {\bf D}   \  = \   
       $ v_2$ $\cdot$   $  \left(   \begin{array}[ ]{cc}
                     b  &    0  \\
                    0   &    1   	
            \end{array} \right) $ \ \ =: \ $ v_2$ $\cdot$ $\widetilde{{\bf D}}$.  \ \
       One has \ \ $\left\|{\bf D}\right\|(s,t)$ = $|v_2|$ $\cdot$ $\|\widetilde{{\bf D}}\|(s,t)$ 
        =  $|v_2| \cdot \sup \{\|\widetilde{{\bf D}}(\vec{x})\|_t \ \ | \ \ \vec{x} \in                                                                           \mathbbm{C}^{2} \  \wedge \ \|\vec{x}\|_s = 1 $ \}.   \ \
       Again we  consider the  map \ \ \
   $ G\:^{t}(y) := (\|\widetilde{{\bf D}}(\vec{x})\|_t)^{t}$ = 
                                         $  y^{t}\cdot |b|^{t} + \left[\sqrt[s]{1-y^{s}} \right]^{t}$,
   ( here for all $ y  $ in the open interval \ $ (0,1) $  ). \ \ 
      As in    {\large \bf Case a},  we have: \ \  $(G\:^{t})'(y_E) = 0   \    \Leftrightarrow      \
                     y_E =    \sqrt[s]{\frac{1}{1+|b|^{\frac{st}{t-s}}}}$ , \ \  which yields   a  minimum 
   for  the map \   $G\:^{t}$, \ \  but  a maximum  for the  map \ $ G $,  \ \  and we get                     
   the maximum  \ \  \   $\max\:\{\|\widetilde{{\bf D}}(\vec{x})\|_t \ \: | \ \:  \vec{x} \in                                                          \mathbbm{C}^{2} \ \ { \rm and } \ \ \|\vec{x}\|_s = 1  \} $ \  = \
              $\max\:\{ G(y) \ | \ y \in [0,1]  \} \ = \ G(y_E)  $.  \\
         As above,  we have   \  \      
           $  G(y_E)    =  \:  [1+|b|^{\frac{s \cdot t}{s-t}}]^{\frac{s-t}{s \cdot t}}  $, \ \  
              and \ \       $\left\|{\bf D}\right\|_{s,t}$ =   $|v_2|$ $\cdot$  \ $G(y_E)$ \ 
     = $ [\: |v_2|^{\frac{s \cdot t}{s-t}} + |v_1|^{\frac{s \cdot t}{s-t}} \: ]^{\frac{s-t}{s \cdot t}}$, \ 
     and     the theorem is proved for $ n = 2 $.    \\  
     Because of \  \ $ t < 0 $, \   \  we  have to  continue  as  in  {\large \bf Case b} , subcase 2.  \\  
      Let  \quad $\vec{x} := ( x_1, \:  \ldots , \: x_{n-1}, x_n )   \in      \mathbbm{C}^{n}$, and  let
        \ \ $ y_1, y_2, \  w_1, w_2  \in     \mathbbm{C}$.   \\ 
     We just have proved that  \ 
     $|y_1 w_1|^{t} +  |y_2 w_2|^{t}$ $\geq$   
     $ [\: |y_1|^{\frac{st}{s-t}} + |y_2|^{\frac{st}{s-t}} \: ]^{\frac{s-t}{s}}$  $\cdot$ 
                                     $\left[ \sqrt[s]{|w_1|^{s}+|w_2|^{s}} \right]^{t}$ holds.  \\ 
     Assuming  the theorem for $ n-1 $, we get                                  
     \  \   
    $\sum_{i=1}^{n-1}|v_ix_i|^{t}$ $\geq$  $ \left[ \sum_{i=1}^{n-1} |v_i|^{\frac{st}{s-t} } \right]^{\frac{s-t}{s}} $           $\cdot$        $\left[ \sqrt[s]{\sum_{i=1}^{n-1} |x_i|^{s}} \right]^{t}$. \\   
      Hence we compute  as four times  before    \\   
    $\left[\|{\bf D}(\vec{x})\|_t\right]^{t}$  = $\sum_{i=1}^{n-1}|v_ix_i|^{t} + |v_nx_n|^{t}$   \
     $ \geq $ 
     $ \left[ \sum_{i=1}^{n-1} |v_i|^{\frac{st}{s-t} } \right ]^{\frac{s-t}{s}} $
                  $  \cdot  \left[ \sqrt[s]{\sum_{i=1}^{n-1}|x_i|^s}  \right]^{t}  + |v_nx_n|^{t}$  \\ 
    $ \geq $  
    $    \left[ \sum_{i=1}^{n-1} |v_i|^{\frac{st}{s-t} } + |v_n|^{\frac{st}{s-t} } \right ]^{\frac{s-t}{s}}  
     \cdot   \left[\sqrt[s]{\sum_{i=1}^{n-1}|x_i|^{s}  + |x_n|^{s }}\right]^{t} $
    = $\left[ \|\vec{v}\|_{\frac{s \cdot t}{s-t}} \right]^{t} \cdot  \|\vec{x}\|_s  ^{t}$ .   \\
    Because of \ $ t <$ 0, this is equivalent to \  $\|{\bf D}(\vec{x})\|_t$ 
                 $\leq$    $ \ \|\vec{v}\|_{\frac{s \cdot t}{s-t}}  \cdot \  \|\vec{x}\|_s$ , \ hence
    \ $\|{\bf D}\|(s,t)$ $\leq$ $\|\vec{v}\|_{\frac{s \cdot t}{s-t}}$.   
     To check equality, one  can use the same vector as two times before, i.e. 
     define   \ \ for \  $ i = 1,2, \ldots n : \ \ r_i := \sqrt[s-t]{|v_i|^{t}}$  , \ \  and   \ \   
   $\vec{z}$  := $ \frac{1}{\sqrt[s]{\sum_{i=1}^{n} |v_i|^{\frac{s \cdot t}{s-t}}}}\cdot ( r_1 , \ \ldots , \ r_n )$.
     \\    \\ 
      {\large \bf Case d)} \quad Let \ \  $  \  -\infty = s \leq t < 0 \: \  {\rm and } \: \ 
      \prod_{i=1}^{n} v_i \neq 0 , \ \ \  { \rm or \ \ let } \ \   \  0 < s \leq t = \infty. \  $   \\    
      If  \quad  $ 0 <  s $ $\leq  t   =   \infty$, \quad    \    
    take \ \ $\vec{e_{\sf M}} := (0, \ldots , 0,1,0,  \ldots , 0 ) , \ \ $ hence $ \ \ 
     \|\vec{e_{\sf M}} \|_{s} = 1 $, \ \ 
   and \ \ \ $ \|{\bf D} (\vec{e_{\sf M}} )\|_{\infty}  =  |v_{\sf M}| $, \ \
    and   \ \  $\|{\bf D}\|$  =   $|v_{\sf M}| = \max \: \{ |v_1| , \ldots , |v_n| \}$  \quad  follows.  \\
       If  \quad \  $-\infty = s \leq t < 0 $ \  and \ $ \prod_{i=1}^{n} v_i \neq 0 $ \  \ 
                          \quad  one can use the vector  \quad   $\vec{e_k}$    \ 
            ( for all $ \: k \in \mathbbm{N}$ ) \ \ with \ \ $ e_{k, \sf M}$ := 1, \ and  \ \ 
         for all \ \: $ i \in \{1, \ldots , n \}\backslash \{ \sf M \} $    
        \  \ $ e_{k,i} := k $,  \ \  hence \  \   
       $\|\vec{e}_k\|_{-\infty} $ = 1, \ \ and \\  
                 $ \lim_{ k \to \infty}$  
                          ($\|{\bf D}(\vec{e}_k)\|_{t})$  =  $ |v_{\sf M}| =   \|\vec{v}\|_{\infty} $ ,  \ \
                              and all  four cases \ {\large \bf Case a} $-$ {\large \bf Case d} \
      are proved, hence CASE $(\mathbb{C})$  \ and \ CASE  $(\mathbb{D})$
         are confirmed.    \\    \\
      It remains to prove  one case of the theorem.  \\  CASE $(\mathbb{E})$.  \ \
    Let     $   \ t = -\infty  \ { \rm and } \ \prod_{i=1}^{n} v_i \neq 0. $    \ \  
       As it has been shown before, the statement is true if  \quad ( $ t = -\infty = s $ ) \  or \      
     ( $t = -\infty   $ and  $  s = \infty$ ).  \quad  So assume   \ \ 
    $ t = -\infty < s \in \ \mathbbm{R}\backslash\{0\} $.  \
      Take a \ \ \  $\widetilde{t} \neq 0 \ \ $ with $ \ \ -\infty <  \widetilde{t} < s $, \ 
       it is already proved that      \ \  
       $ \|{\bf D}\|_{s,\widetilde{t}} =    \|\vec{v}\|_{\frac{s \cdot \widetilde{t}}{s-\widetilde{t}}} $ . \ \
        Thus \\  $ \|{\bf D}\|_{s,-\infty}$ \ = \
         $\lim_{ \widetilde{t} \rightarrow -\infty}$              
       [ $ \|{\bf D}\|_{s,\widetilde{t}}$ ] \ = \ 
         $\lim_{ \widetilde{t} \rightarrow -\infty}$ 
         [ $ \|\vec{v}\|_{\frac{s \cdot \widetilde{t}}{s-\widetilde{t}}} $ \ ] = \ $\|\vec{v}\|_{-s} $ .  \\
       For equality one takes the vector \ \   $\vec{z}$ \ := \  \ 
      $ \|\vec{v}\|_{-s} \cdot  ( \frac{1}{v_1} , \ldots , \frac{1}{v_n} ) $  =  
        $ \sqrt[-s] {\sum_{i=1}^{n}  \frac{1}{|v_i|^{s}}}  \cdot  ( \frac{1}{v_1} , \ldots , \frac{1}{v_n} ) $,    
           \\  hence  \ \ $\|\vec{z}\|_{s} = 1 $ \ \ and \ \ $\|{\bf D}(\vec{z})\|_{-\infty} =  \|\vec{v}\|_{-s}$ ,
            \quad      and the proof of \  {\bf Theorem 1} \ is finished.   \\
         \section{Proofs of   Theorem 2 and the Corollaries} 
    The  { \bf Corollary 1}  follows immediately by observing that   \\
    $\frac{s \cdot t}{s-t}$  =  $\frac{(-t) \cdot (-s)}{(-t)-(-s)}$ ,  \quad
                 and \ \ \ $ s \leq  t \ \Longleftrightarrow \ -t \leq -s $.  \\  \\
      Before we can prove \ {  {\bf Theorem 2} } \   we mention a fact, which is easy to confirm.               
    \begin{fact}
    \quad Let \ \ \  $ r, s, t \in  \mathbbm{R}$, \ \ \ such that \ \   $ 0 \neq$ $r \cdot s \cdot t$  
    \ \  and  \  \   $\frac{1}{t}$ = $\frac{1}{r} \ + \ \frac{1}{s}$ .  \\
    Then   \  either \   \qquad  t $<$ r, s \ \ \ or \ \ \  t $>$  r, s.   \\ 
  If  furthermore   \ \quad \ $ t < 0, \ r, \ s $ \ \ or \ \ $ t >  0, \ r, \ s $ ,
                                        \quad  then \quad $ r \cdot s < 0 $ .   
   \end{fact}   
     Now we are able to prove   \ {{\bf Theorem 2}}.
   \begin{proof}                   
     This theorem  is trivial if \ $ n = 1$. \   So let  \   $ n > 1 $. \
  Let   \ \ $ t < r , s $. \ Now take the  {\bf Theorem 1},  CASE $(\mathbb{C})$,  \  
                                  and note that \ $ r = \frac{s \cdot t}{s-t}$ .  \\ 
     Let    \quad \  $ t > r , s $ . \ \
  In the case of \   $ \| \vec{v}  \|_r $ $\cdot$  $ \| \vec{x} \|_s $  =  0, \ the inequality holds. \ 
  Hence assume \\   $ \| \vec{v}  \|_r $ $\cdot$  $ \| \vec{x} \|_s  \neq 0 $.
   Because of {\bf Fact 1}  and \ \   $\frac{1}{t}$ = $\frac{1}{r} \ + \ \frac{1}{s}$ , \ \ 
     three cases are possible, 
     namely \ \quad  $ 0 > t > r , s , \ \ $ or $ \ \ t > r > 0 > s , \ \ $ or $ \ \ t > s > 0 > r $.   \\
     In the first two cases \ $s$ is negative, and because of  \  $ \| \vec{x} \|_s  \neq 0 $, \ \
      $ x_i \neq 0 $ \ \ holds for every $i$.   \ \
      One has  \ \ \  $\frac{1}{r}$ = $\frac{1}{t} \ + \ \frac{1}{-s}$  \ \ and \ (with  {\bf Fact 1} ) \
                                           $ r < t , -s $. \ \
      Let \ \  for all \ $ i \in \{1, \ldots , n \}$: \ \  $\widetilde{x_i}$ :=  $v_i \cdot x_i$  \\
               and \ $z_i$ := $\frac{1}{x_i}$ .  \ \
      Because of \ \  $ r < t  , -s $ \ \  we get \\  
       $ \sqrt[r]{\sum_{i=1}^{n}|\widetilde{x_i} \cdot z_i|^{r}}$ $\leq$ 
          $\sqrt[t]{\sum_{i=1}^{n}|\widetilde{x_i}|^{t}}$ $\cdot$  $\sqrt[-s]{\sum_{i=1}^{n}|z_i|^{-s}}$  \                $\Longleftrightarrow$ \  $\sqrt[r]{\sum_{i=1}^{n}|\widetilde{x_i} \cdot z_i|^{r}}$  $\cdot$  
           $\sqrt[+s]{\sum_{i=1}^{n}|z_i|^{-s}}$  $\leq$   $\sqrt[t]{\sum_{i=1}^{n}|\widetilde{x_i}|^{t}}$   \         
      $\Longleftrightarrow$ \  $\sqrt[r]{\sum_{i=1}^{n}|v_i|^{r}}$  $\cdot$  
           $\sqrt[s]{\sum_{i=1}^{n}|x_i|^{s}}$  $\leq$   $\sqrt[t]{\sum_{i=1}^{n}|v_i \cdot x_i|^{t}}$  
        $\Longleftrightarrow$ \ \    $ \| \vec{v}  \|_r $ $\cdot$  $ \| \vec{x} \|_s $  \
         $\leq$  $ \| \vec{v} \cdot \vec{x} \|_t $ .
                               \\  
        The remaining last case \ \  $  t > s > 0 > r $   \ \  is treated in the same way:  \ \
        because of \ \ $ 0 > r $ \ \ and   \ \   $ \| \vec{v} \|_r  \neq 0 $, \ \
      $ v_i \neq 0 $ \ \ holds for every $i$.   \ \ 
      Hence define \  
           for all \  $ i \in \{1, \ldots , n \}$: \  \  $\widetilde{x_i}$ :=  $v_i \cdot x_i$ 
              \ and \ $z_i$ := $\frac{1}{v_i}$ , \ \  and then  one can go the same way as only just.
          This finishes the proof.  
         \end{proof}
       The  { \bf Corollary 2}  follows directly from   { \bf Theorem 2}.       
          \begin{remark}    \rm
          However, this version  of the Hölder-inequality is not realy an extension, but  
       equivalent  with the       usual one \ 
       (  $1 = \frac{1}{r} \ + \ \frac{1}{s} \  $  and $ \ 1 < r,s \     \Longrightarrow  $ \    
         $  \| \vec{v} \cdot \vec{x} \|_1 $   $\leq$ 
                                        $ \| \vec{v}  \|_r $ $\cdot$  $ \| \vec{x} \|_s $ ).  
       \\   For positive values of \ $ r, s, t $ \ one can find a short proof in \   \cite{Meise/Vogt},p.103.
       \ The general case  which  includes negative values  is treated in the next section. 
        \end{remark} 
        \section{Measurable Functions}
        In this last section we demonstrate that the generalized Hölder inequality also holds in the
        $ {\cal L}^{\: p} $ function spaces. The proofs rely mainly on the standard Hölder inequality.   
           At first we have to define the \ $ {\cal L}^{\: p} $ spaces  also for negative $ p$.  \\  
                 Let \quad  $ (\Omega, {\cal A} , \mu ) $ \  be a measure space  \  with
         \ \   $ \ \mu(\Omega)  > 0 $. \  We use the conventions \   $ \infty \cdot 0 := 0 $   \ and  \
                 $ \frac{1}{0}  := \infty  $.  \ \  
      Let  \quad  $ {\cal M}_{\Omega} :=   \{ \: f : (\Omega, {\cal A} , \mu ) \rightarrow  \mathbbm{R}  
       \cup \{ -\infty, \infty \} \ | \ f { \rm  \: is \: measurable \: } \}$.  \ \    
     Define   \ for every  \  \ $p < 0$: \  \
       $ {\cal L}^{\: p} :=  {\cal L}^{\infty} :=
    \{  f \in {\cal M}_{\Omega} \ |
    \  \ { \rm ess \ sup} \ \{ |f(\omega)| \ | \: \omega \in \Omega \} < \infty  \ \} $. \\ 
        Then we define \     for all \  $ f \  \in {\cal M}_{\Omega}  $ 
            \[  \| f \|_{p}  :=             
      \begin{cases}   
              \sqrt[p]{\int_\Omega {|f|^{p}} \: d\mu }  \:     \quad  &
       \quad \Longleftrightarrow \quad  0 < \int_\Omega |f|^{p} \: d\mu     < \infty   \\    
         0   &  \quad \Longleftrightarrow \quad  \int_\Omega |f|^{p} \: d\mu    = \infty     \\ 
        \infty  &  \quad \Longleftrightarrow \quad  \int_\Omega |f|^{p} \: d\mu    = 0   \\               
      \end{cases}                                                                   \]   
     Note that  \ for  \   $ f  \in {\cal L}^{\infty} , \      \| f \|_{p}  \ < \ \infty $ \ holds. \                      And  for every \  $p > 0$ \ we take the usual definition,        \
 $ {\cal L}^{\: p}  :=  \{ f: (\Omega, {\cal A} , \mu ) \rightarrow  \mathbbm{R}    \cup \{ -\infty, \infty \} \   
    | \ f \: \in {\cal M}_{\Omega} \ \ { \rm and }  \ \  \int_\Omega {|f|^{p}} \: d\mu < \infty   \}, $ \  
    and \  for all  \ $ f \in {\cal M}_{\Omega}$ take 
  \ \  $ \ \  \| f \|_{p}  :=   \sqrt[p]{ \int_\Omega {|f|^{p}} \: d\mu }$  .  \\
  By making an equivalence relation $N$ ( $ f \approx_N g $ \ $\Leftrightarrow$ \ $ f, g $ 
                           distinguish only on a zero set),
  and  by defining \  $ \bf M_{\Omega} := {\cal M}_{\Omega}/_{\approx_N} $,  \ and \
   for all $p \ \in \mathbbm{R}\backslash\{0\}:  \quad \  {\bf L}^{p} := {\cal L}^{\: p}/_{\approx_N}$,    
  \ \ this definition makes  that the pairs  \ \ 
     \  ( ${\cal M}_{\Omega} , \|..\|_{p} $ ), \  ( ${\bf M_{\Omega}}, \|..\|_{p} $ ), \  
   $( {\cal L}^{p} , \|..\|_{p} ) $ \  and \   $({ \bf L}^{p} , \|..\|_{p} ) $ \  are  {\bf hw spaces}
                             \   for all  $ p \in \mathbbm{R}\backslash\{0\}$.  \ \
      These  homogeneous weights   \ $ \|..\|_{p} $ \ we call  the   Hölder weights   on  ${\cal M}_{\Omega}$,
    ${\cal L}^{p}$, ${\bf M_{\Omega}}$ \ or \ $\bf L^{p} $, respectively.  \ \                         
    It is known that \ $( {\bf L}^{p} , \|..\|_{p} ) $ \ is a pseudonormed space if and only if \ \ $p > 0$, \ 
     and it is a  normed space if and only if \ $ p \geq$  1.  \\ 
      Now let us recall the   well-known Hölder inequality and the  reverse Hölder inequality  
      for  measurable functions.  For    
      two real numbers   $ r, s$  \ such that  \  1 $ < $ $r , s $  \  
                          and \   $ 1 \ =  \ \frac{1}{r} \ + \ \frac{1}{s}$ , \ 
              we have   \ for all    
                 measurable  functions \: $ f, g $  \ \ ( that means \ $ f , g $   $ \in {\cal M}_{\Omega}$ ): \ \
                \    $ \| f \cdot g \|_1 $   $\leq$   $ \| f  \|_r $ $\cdot$  $ \| g \|_s $  .   \\
    For the next  inequality  see e.g.  \cite{Elstrodt},p.226, 
     or   \cite{Mitrinovic/Vasic},p.51,  or     \cite{Hewitt/Stromberg},p.191.
     \begin{corollary}     
     \quad  Let \quad   $ r, s \in  \mathbbm{R} \backslash \{0\} $ \ \ \ such that \ \ $ 1  >  r , s $  \ \ 
      and  \ \  $ 1 \ =  \ \frac{1}{r} \ + \ \frac{1}{s}$ . \\ 
                $ ($  Hence \ \  either \quad  $ r < 0 < s  \quad or  \quad  s < 0 < r \ )$.  \\         
               Then one has  \ for all    
                 measurable  functions \  $ f , g , $ \ that a reverse Hölder inequality holds, i.e.    
             \[      \| f \cdot g \|_1    \geq 
                              \| f  \|_r  \cdot   \| g \|_s \ .   \]
         \end{corollary} 
     \begin{proof}  \quad    
    Assume \quad  $ r < 0 < s < 1 $.    \ \
   Now we have to distinguish  three cases. \\
   1) \ $ \| f  \|_r $ = 0. \quad \quad  \ \ The inequality holds. \ ( Note that  $ \infty \cdot 0 = 0 ) $. \\
   2) \ $ \| f  \|_r  = \infty . $  \quad \quad  We have  \\
    $ \| f  \|_r  = \infty  \ \  \Longleftrightarrow \ \ 
     {\int_\Omega |f|^{r}  \ d\mu} = 0 
     \ \  \Longleftrightarrow \ \  |f|(\omega) = \infty \ $ (for \ almost \ all \   $ \omega \in \Omega) . $   \\
   In the case of \ \    $ \| g \|_s $ = 0, \ \  the inequality holds.  \ \
   In the case of \ \    $ \| g \|_s  > 0 $, \   there is a measurable set A  with \  A $ \subset \Omega $, 
   \    and \  $   \mu(A) > 0 \   $ and $ \   
  |g|(\omega) > 0 \ (\forall \: \omega \in A),  $ hence it follows
    $  \ |f \cdot g|(\omega) = \infty  \ $ (for  almost  all  \: $ \omega \in A) $,  \
   hence \ \   $ \| f \cdot g \|_1 $  = $\infty $.  \\
   3) \ $ 0 < \| f  \|_r  < \infty . $    \\
   We have \quad $ \frac{1}{s} \ =  \ 1 \: + \:  \frac{1}{-r}    $ , \ \ hence  \ 
     $  1 =   \frac{1}{1/s} +   \frac{1}{-r/s} \:  \  $ and $ \  $ (with  $ {\bf Fact \: 1} ) \ 
       \  1 < \frac{1}{s} , \: \frac{-r}{s} \: . $  \ \ 
     Define   \ \   $  v := |f|^{s} \cdot |g|^{s} $,  \  $ w :=  |f|^{-s} , $ \     hence  \  
     $ v , w  \  \ \in {\cal M}_{\Omega}$,    \  
     and  we have  by the Hölder inequality  \ \  
                              ( note that  \ \ $ 0 < {\int_\Omega  |w|^{\frac{-r}{s}}  \ d\mu} < \infty \: ) $    
     \\   $ \| v \cdot w \|_1 $ $\leq$ $ \| v  \|_{\frac{1}{s}}  $ $\cdot$  $ \| w \|_{\frac{-r}{s}} $ 
     \quad $ \Longleftrightarrow $       \quad 
  $  \sqrt[s]{\int_\Omega  |v \cdot w| \ d\mu} \leq  
           {\int_\Omega  |v|^{\frac{1}{s} } \ d\mu} \ \cdot \   \sqrt[-r]{\int_\Omega  |w|^{\frac{-r}{s}}  \ d\mu} $ \\
   \quad $ \Longleftrightarrow $       \quad 
  $  \sqrt[s]{\int_\Omega  |v \cdot w| \ d\mu} \cdot  \sqrt[+r]{\int_\Omega  |w|^{\frac{-r}{s}}  \ d\mu}  \leq  
           {\int_\Omega  |v|^{\frac{1}{s} } \ d\mu}  \quad  \Longleftrightarrow   \quad            
    \| g  \|_s \cdot \| f \|_r   \leq  \| f \cdot g \|_1 $   \\
   and all three cases of  { \bf Corollary 3 }  has been proved.  
    \end{proof}    
      Now we are able to  formulate the generalized Hölder inequality for measurable functions.  
    \begin{theorem}   
     \quad Let \quad   $ r, s, t \in  \mathbbm{R}$ \ \ \ such that \ \ \  $ 0 \neq r \cdot s \cdot t$  \ \ \
                          and \ \ \ $\frac{1}{t}$ = $\frac{1}{r} \ + \ \frac{1}{s}$ .  \\
           Then we have   \ for all    
               $  \ f , g    \in {\cal M}_{\Omega} \ $   
             \[    t < r , s  \quad \Longrightarrow  \quad     
                                \| f \cdot g \|_t    \leq 
                              \| f  \|_r  \cdot   \| g \|_s \ . \]  
                   \[  t > r , s  \quad \Longrightarrow  \quad     
                               \| f \cdot g \|_t    \geq 
                              \| f  \|_r  \cdot   \| g \|_s \ .      \]     
   \end{theorem}  
   \begin{proof}
    The proof is inspired by  \  \cite{Meise/Vogt},p.103. \ \ 
                                We distinguish four cases.    \\
   1)   $t < r , s  \quad  { \rm and } \quad  t  > 0 $ \quad  \quad \quad  \quad 
   2)   $t < r , s \quad  { \rm and } \quad t < 0 $ \\
   3)   $t > r , s  \quad { \rm and } \quad  t  > 0 $  \quad  \quad  \quad \quad 
   4)   $t > r , s \quad  { \rm and } \quad t  < 0$  \\                              
      We  only show case 2.  All the other cases follow along the same lines.   \\ 
       \quad Let \ \ $t < r , s \quad  { \rm and } \quad  t < 0$. \\
   Let \quad  $   f , g$   $ \in {\cal M}_{\Omega}$ . \quad Then define \ $v, w $  $ \in {\cal M}_{\Omega}$ ,
   by taking \ $v := |f|^{t}$ ,  $ w := |g|^{t}$  . \\
   Because of \ \  $ 1 = \frac{1}{r/t} + \frac{1}{s/t}$ , \ \ and \ $ 1 >  \frac{r}{t} , \frac{s}{t}  $ , \
    and because of the previous  {\bf Corollary \:3},  \ we have \ \quad
       $ \| v \cdot w \|_1 $   $\geq$ 
                             $ \| v  \|_\frac{r}{t} $ $\cdot$  $ \| w \|_\frac{s}{t} $  \ \    
      $\Longleftrightarrow$  \ \           $ \| f \cdot g \|_t $ $\leq$ $ \| f  \|_r $ $\cdot$  $ \| g \|_s $ . 
   \end{proof}   
      \newpage      
      Acknowledgements:  \\    
     The author  thanks \ Prof.  Dr.  Marc Keßeböhmer,  Dr. Björn  Rüffer and  
                                   Dr.  Gencho Skordev for support and help.     
  
              \nopagebreak                                                           
                                          
    \end{document}